\documentclass[a4 paper,12pt,bezier]{article}
\usepackage[top=3cm,bottom=3cm,left=2.5cm,right=2.5cm]{geometry}
\usepackage{amsmath}
\usepackage{amsfonts,amsthm,amssymb}
\usepackage{amsfonts}
\usepackage{graphics,subfigure,caption2}
\usepackage{tikz, color}
\usepackage{graphics,epstopdf}
\usepackage{latexsym,bm}
\usepackage{amsfonts,amsthm,amssymb,mathrsfs,bbding}
\usepackage{hyperref}
\usepackage{CJK}
\usepackage{color}
\usepackage{indentfirst}
\usepackage{graphicx}
\usepackage{cleveref}
\usepackage{booktabs}
\usepackage{multirow}
\usepackage{cleveref}

\newtheorem{lem}{Lemma}[section]
\newtheorem{nota}{Notation}
\newtheorem{thm}[lem]{Theorem}
\newtheorem{cor}[lem]{Corollary}

\newtheorem{df}[lem]{Definition}

\newtheorem{prob}{Problem}
\newtheorem{cla}[lem]{Claim}
\theoremstyle{plain}

\usepackage{cite}

\begin{document}
\begin{CJK}{GBK}{song}
\title{Extremal problems on disjoint path covers of graphs}
\author{Shujie Chen, Tao Tian\footnote{Corresponding author. E-mail: shujiechen2024@163.com (S. Chen), taotian0118@163.com (T. Tian).}\\
\small  School of Mathematics and Statistics, Key Laboratory of Analytical Mathematics and\\
\small  Applications $($Ministry of Education$)$, Fujian Key Laboratory of Analytical Mathematics and\\
\small  Applications (FJKLAMA), Center for Applied Mathematics of Fujian Province $($FJNU$)$,\\
\small Fujian Normal University,
\small Fuzhou 350117, PR China.}
\date{}
\maketitle

\maketitle {\flushleft\bf Abstract}:
In 1962, Erd\H{o}s characterized the maximum size of nonhamiltonian graphs of order $n$ with minimum degree at least $k$. Later, Ning and Peng [Combin. Probab. Comput. 29 (2020) 128-136] extended Erd\H{o}s's results to the clique condition and provided the maximum clique number for nonhamiltonian graphs of order $n$ with minimum degree at least $k$. Recently, Zhang [European J. Combin. 112 (2023) 103728] determined the maximum number of $s$-cliques in nonhamiltonian graphs with prescribed order and minimum degree. A natural extension is to characterize the maximum number of $s$-cliques under other graph properties. Notably, disjoint path cover problems are closely related to Hamiltonicity. In this paper, we generalize results on Hamiltonicity and establish sufficient conditions for a graph to possess one-to-one, one-to-many and many-to-many $t$-disjoint path covers in terms of the number of cliques and the $\alpha$-spectral radius, respectively. Furthermore, we characterize the extremal graphs that attain these bounds respectively.

\maketitle {\flushleft\textit{\bf Keywords}}: Clique, $\alpha$-spectral radius, Disjoint path cover, Hamiltonian, Hamilton-connected

\section{  Introduction  }\label{sec1}
In this paper, we only consider finite, connected and simple graphs.  We use $V(G)$ and $E(G)$ to denote the vertex set and the edge set of a graph $G$, respectively. Let $|V(G)|$ (or $|G|$) and $e(G)$ denote the \emph{order} and \emph{size} of a graph $G$, respectively. For $v\in V(G)$, we let $d_G(v)$ denote the \emph{degree} of $v$ in $G$. We use $\delta(G)$ to denote the minimum degree of $G$ (or simply $\delta$ if no confusion can arise). For a vertex subset $S\subseteq V(G)$, the induced subgraph $G[V(G)\backslash S]$ is denoted by $G-S$. If $S=\{v\}$, we write $G-v$ for $G-\{v\}$. A subset $X\subseteq V(G)$ is called a \emph{vertex cut} of $G$ if $G-X$ has at least two components. The \emph{connectivity} of graph $G$,  denoted by $\kappa(G)$, is the cardinality of a minimum vertex cut of a graph $G$. Define $\kappa(K_{n})=n-1$, where $K_{n}$ is a complete graph of order $n$. A graph $G$ is said to be \emph{$k$-connected} if $\kappa(G) \geq k$. A graph with no vertices is a \emph{null graph}. For two disjoint graphs $G$ and $H$, denote by $G\lor H$ and $G\cup H$ the \emph{join} and the \emph{union} of $G$ and $H$, respectively. For any two distinct vertices $u,v\in V(G)$ with $uv\notin E(G)$, we denote by $G+uv$ the graph obtained from $G$ by adding an edge $uv$ to $G$. Let $\overline{G}$ denote the \emph{complement} of $G$. A set of vertices $W\subseteq V(G)$ is a \emph{clique} of a graph $G$ if for any two vertices $x,y\in W$, $xy\in E(G)$. An \emph{$s$-clique} is a clique with cardinality $s$. We use $N_s(G)$ to denote the number of $s$-cliques in $G$. The \emph{circumference} of a graph $G$ is the length of a longest cycle in $G$. For a  graph $G$, we define
\begin{itemize}
	\item$\delta(G)=\emph{\rm{min}} \{d_G(v) \mid v\in V(G) \} \; \emph{\rm{(Dirac-type)}};$
	\item${\sigma}_2(G)=\emph{\rm{min}} \{d_G(x)+d_G(y) \mid x,y \in V(G), \; xy\notin E(G) \} \; \emph{\rm{(Ore-type)}};$
\end{itemize}
If $G$ is a complete graph, we define $\sigma_2(G)=\infty$. For other terms not explained in this paper, the reader can refer to \cite{a1}.

Given a graph $G$ with vertex set $V(G)=\{v_1,v_2,\dots,v_n\}$, the adjacency matrix of $G$ is denoted by $A(G)$. The $(i,j)$-entry of $A(G)$ is $1$ if $v_iv_j\in E(G)$, and $0$ otherwise. The diagonal matrix of vertex degrees of $G$ is denoted by $D(G)$. The Laplacian matrix $L(G)$ and signless Laplacian matrix $Q(G)$ of $G$ are defined by $L(G)=D(G)-A(G)$ and $Q(G)=D(G)+A(G)$, respectively. The spectral radius of $G$ is the largest eigenvalue of $A(G)$, denoted by $\lambda(G)$. The signless Laplacian spectral radius of the graph $G$ is the largest eigenvalue of $Q(G)$,  denoted by $\mu(G)$. For $\alpha\in [0,1)$, Nikiforov \cite{a2} defined the $A_\alpha$-matrix of $G$ as
$$
A_\alpha(G)=\alpha D(G)+(1-\alpha)A(G).
$$
The \emph{$\alpha$-spectral radius} of $G$ is the largest eigenvalue of $A_\alpha (G)$, denoted by $\lambda_\alpha(G)$. It is obvious that $\lambda_0(G)$ is the spectral radius of $G$, and $2\lambda_{1/2}(G)$ is the signless Laplacian spectral radius of $G$.

A \emph{path cover} of a graph $G$ is a set of paths in $G$ such that every vertex of $G$ is contained in at least one path. A \emph{disjoint path cover} of $G$ is a set of paths in $G$ such that every vertex of $G$ belongs to exactly one path. Given two disjoint vertex subsets $U=\{u_1,u_2,\dots,u_t\}$ and $V=\{v_1,v_2,\dots,v_t\}$ of $V(G)$, the \emph{many-to-many $t$-disjoint path cover} of $G$ is defined as a disjoint path cover of size $t$, where each  path connects $u_i$ and $v_{\tau(i)}$ for some  permutation $\tau$ of $\{1,2,\dots,t\}$. A graph is \emph{many-to-many $t$-disjoint path coverable} if there exists a many-to-many $t$-disjoint path cover for any two disjoint vertex subsets $U=\{u_1,u_2,\dots,u_t\}$ and $V=\{v_1,v_2,\dots,v_t\}$.  By modifying the sets $U$ and $V$ to $U=\{u\}$ and $V=\{v_1,v_2,\dots,v_t\}$ or $U=\{u\}$ and $V=\{v\}$, we can obtain the definitions of \emph{one-to-many t-disjoint path coverable} and \emph{one-to-one $t$-disjoint path coverable}.

The disjoint path cover problem, a fundamental topic in graph theory, finds crucial applications in software verification, database design, and code optimization \cite{a3,a4}. It focuses on constructing vertex-disjoint paths in interconnection networks to maximize node utilization. Extensive research has been conducted on this problem (see \cite{a5,a6,a7}). Determining whether there exists a one-to-one, one-to-many, or many-to-many t-disjoint path cover for a given pair of terminal sets in a graph is NP-complete \cite{a8,a9}.

A graph $G$ is called \emph{Hamiltonian} if $G$ contains a Hamilton cycle, i.e., a cycle that contains every vertex of $G$. A \emph{Hamilton path} in a graph $G$ is a path that contains all the vertices of $G$. A graph $G$ is said to be \emph{Hamilton-connected} if every pair of distinct vertices in $G$ are connected by a Hamilton path. Indeed, the disjoint path cover problem is inherently connected to connectivity and Hamiltonicity, which are two fundamental concepts in graph theory. Clearly, one-to-one and one-to-many $t$-disjoint path coverable graphs are $t$-connected. Tian \cite{a10} proved that if a graph $G$ is one-to-many $t$-disjoint path coverable, then $G$ is $(t+1)$-connected. A graph with at least 3 vertices is one-to-one 2-disjoint path coverable if and only if it is Hamiltonian. Furthermore, a graph with at least 3 vertices is one-to-one 1-disjoint path coverable or one-to-many 2-disjoint path coverable or many-to-many 1-disjoint path coverable if and only if it is Hamilton-connected.

Determining sufficient conditions for the Hamiltonicity of a graph remains a hot topic in graph theory. Therefore, the close relationship between disjoint path covers and Hamiltonicity motivates the investigation of sufficient conditions for graphs to be one-to-one, one-to-many, or many-to-many $t$-disjoint path coverable.

Many classic results have been established regarding the conditions for a graph to be Hamiltonian (see \cite{a11,a12,a14}). Motivated by these results, sufficient conditions for disjoint path covers have also been extensively studied, including Dirac-type conditions \cite{a18,a19}, Ore-type conditions \cite{a18,a20}, P\'{o}sa-type and Bondy-type conditions \cite{a23}, among others.

\begin{thm}\label{thm0}
	Let $G$ be a graph of order $n$ and $t\geq 2$.
	\begin{enumerate}
		\item[(i)] \emph{\rm{(Lin et al. \cite{a18})}} If $\delta(G)\geq \frac{n+t-2}{2}$, then $G$ is one-to-one $t$-disjoint path coverable.
		\item[(ii)] \emph{\rm{(Lin et al. \cite{a19})}} If $\delta(G)\geq \frac{n+t-1}{2}$, then $G$ is one-to-many $t$-disjoint path coverable.
		\item[(iii)] \emph{\rm{(Lin et al. \cite{a19})}} If $\delta(G)\geq \frac{n+t}{2}$, then $G$ is many-to-many $t$-disjoint path coverable.
	\end{enumerate}
\end{thm}

In particular, determining the minimum number of edges ensuring Hamiltonicity remains a prominent research topic. In 1961, Ore \cite{a24} derived the maximum size of a nonhamiltonian graph with a given order and also determined the extremal graphs.

\begin{thm}\emph{\rm{(Ore \cite{a24})}}
	The maximum size of a nonhamiltonian graph of order $n\geq 3$ is $\binom{n-1}{2}$ and this size is attained by a graph $G$ if and only if $G=K_1\vee (K_{n-2}\cup K_1)$ or $G=K_2\vee \overline{K}_3$.
\end{thm}

A natural extension is to investigate the problem under additional graph constraints. In 1962, Erd\H{o}s \cite{a16} derived the maximum size of a nonhamiltonian graph of order $n$ with minimum degree at least $k$.

\begin{thm}\emph{\rm{(Erd\H{o}s \cite{a16})}}
	Let $n$, $k$ be integers with $1\leq k\leq \left\lfloor \frac{n-1}{2} \right\rfloor$. If $G$ is a nonhamiltonian graph of order $n$ with minimum degree at least $k$, then
	$$
	e(G)\leq \max \left\{\binom{n-k}{2}+k^2, \binom{n-\left\lfloor \frac{n-1}{2}\right\rfloor}{2}+\left\lfloor \frac{n-1}{2} \right\rfloor^2\right\}.
	$$
\end{thm}

A graph $G$ is \emph{claw-free} if it does not contain $K_{1,3}$ as an induced subgraph. In 2018, Li, Ning and Peng \cite{a34} established the maximum number of edges in a claw-free graph of order $n$ and minimum degree at least $k$ such that the graph is nonhamiltonian.

\begin{thm}\emph{\rm{(Li, Ning and Peng \cite{a34})}}
	Let $k\geq 3$ and $n\geq k^2+8k+4$. If $G$ is a $2$-connected claw-free nonhamiltonian graph of order $n$ with minimum degree at least $k$, then
	$$
	e(G)\leq \binom{n-2k-2}{2}+2\binom{k+1}{2}+6.
	$$
\end{thm}

Since $N_2(G)=e(G)$, the clique condition $N_s(G)$ for a graph to have a certain property has also attracted scholars' attention. In \cite{a35}, Luo  proved the maximum number of $s$-cliques in a $2$-connected  graph of order $n$ with circumference less than $k$.

\begin{thm}\emph{\rm{(Luo \cite{a35})}}
	Let $n\geq k\geq 5$ and $s\geq 2$. If $G$ is a $2$-connected graph of order $n$ with circumference less than $k$, then
	$$
	N_s(G)\leq \max \left\{f_s(n,k,2),f_s\left(n,k,\left\lfloor \frac{k-1}{2} \right\rfloor\right)\right\},
	$$
	where $f_s(n,k,a)=\binom{k-a}{s}+(n-k+a)\binom{a}{s-1}$.
\end{thm}

By imposing minimum degree as a new parameter, F\"{u}redi, Kostochka and Luo \cite{a26} derived the maximum number of $s$-cliques in nonhamiltonian graphs of order $n$ with minimum degree at least $k$.

\begin{thm}\emph{\rm{(F\"{u}redi, Kostochka and Luo \cite{a26})}}
	Let $n,k,s$ be integers with $1\leq k\leq \left\lfloor \frac{n-1}{2} \right\rfloor$ and $s\geq 2$. If $G$ is a nonhamiltonian graph of order $n$ with minimum degree at least $k$, then
	$$
	N_s(G)\leq \max \left\{h_s(n,k),h_s\left(n,\left\lfloor \frac{n-1}{2} \right\rfloor\right)\right\},
	$$
	where $h_s(n,x)=\binom{n-x}{s}+x\binom{x}{s-1}$.
\end{thm}

In 2020, Ning and Peng \cite{a36} derived the maximum number of $s$-cliques in 2-connected graphs of order $n$ with circumference $c$ and minimum degree at least $k$.

\begin{thm}\emph{\rm{(Ning and Peng \cite{a36})}}
	Let $c\leq n-1$ and $s\geq 2$. If $G$ is a $2$-connected graph of order $n$  with circumference $c$ and minimum degree at least $k$, then
	$$
	N_s(G)\leq \max \left\{f_s(n,c,k),f_s\left(n,c,\left\lfloor \frac{c}{2} \right\rfloor \right) \right\},
	$$
	where $f_s(n,c,x)=\binom{c+1-x}{s}+(n-c-1+x)\binom{x}{s-1}$.
\end{thm}

To determine the maximum number of $s$-cliques in a $2$-connected graph with prescribed order, circumference and minimum degree, Zhang \cite{a27} studied the maximum number of $s$-cliques in a $2$-connected nonhamiltonian graph of order $n$ with circumference $c$ and minimum degree $\delta$.

\begin{thm}\emph{\rm{(Zhang \cite{a27})}}
Let $\varphi_s(n,c,\delta)$ denote the maximum number of $s$-cliques in a $2$-connected nonhamiltonian graph of order $n$ with circumference $c$ and minimum degree $\delta$. Then
$$
\varphi_s(n,c,\delta)= {\rm max}\{f_s(n,c,\delta),g_s(n,c,\delta)\},
$$
where $f_s(n,c,\delta)=\binom{c+1-\delta}{s}+(n-c-1+\delta)\binom{\delta}{s-1}$ and $g_s(n,c,\delta)=\binom{c+1-\left\lfloor \frac{c}{2}\right\rfloor}{s}+(n-c-2+\left\lfloor \frac{c}{2}\right\rfloor)\binom{\left\lfloor \frac{c}{2}\right\rfloor}{s-1}+\binom{\delta}{s-1}$.
\end{thm}

In addition, Zhang \cite{a28} studied the maximum number of $s$-cliques in a nonhamilton-connected graph of order $n$ with minimum degree $\delta$. Let $S$ be a subset of $V(G)$. A graph $G$ is \emph{$S$-leaf-connected} if $G$ has a spanning tree $T$ such that $S$ is the set of leaves of $T$. A graph $G$ of order at least $k+1$ is \emph{$k$-leaf-connected} if for any subset $S$ of $V(G)$ with $|S|=k$, $G$ is $S$-leaf-connected. Wu, Xue and Kang \cite{a29} provided a clique condition for a graph being $k$-leaf-connected.

\begin{thm}\emph{\rm{(Wu, Xue and Kang \cite{a29})}}
Let $s,k,\delta$ be integers with $s\geq 2$, $k\geq 2$ and $k+1\leq \delta \leq \frac{n+k-2}{2}$. Suppose that $G$ is a graph of order $n$ with minimum degree at least $\delta$ such that
$$
N_s(G)>\max \left\{h_s(n,k,\delta+1),h_s\left(n,k,\left\lfloor \frac{n+k-2}{2}\right\rfloor \right) \right\},
$$
where $h_s(n,k,\delta)=\binom{n+k-\delta-1}{s}+(\delta+1-k)\binom{\delta}{s-1}$.
\begin{enumerate}
	\item[(i)] If $k=2$, then $G$ is $k$-leaf-connected unless $G$ is a subgraph of $K_\delta\vee (K_{n-2\delta+1}\cup \overline{K}_{\delta-1})$ or $K_2\vee (K_{n-\delta-1}\cup K_{\delta-1})$.
	\item[(ii)] If $k\geq 3$, then $G$ is $k$-leaf-connected unless $G$ is a subgraph of $K_k\vee (K_{n-\delta-1}\cup K_{\delta+1-k})$.
\end{enumerate}
\end{thm}

Spectral conditions are well-known to be a classical approach to Hamiltonian problems in graph theory. In 2010, Fiedler and Nikiforov \cite{a30} derived some sufficient conditions on the spectral radius for the existence of Hamilton cycles. In 2017, Zhou and Wang \cite{a31} provided some sufficient conditions for a graph to be Hamilton-connected in terms of the spectral radius and the signless Laplacian spectral radius of the graph. Ao et al. \cite{a32} extended the results of Zhou and Wang to $k$-leaf-connectivity. Recently, Wu, Xue and Kang \cite{a29} derived some sufficient conditions for graphs to be $k$-leaf-connected in terms of the $\alpha$-spectral radius of the graph.

Motivated by the above results, we naturally consider the clique conditions $N_s(G)$ and the $\alpha$-spectral conditions $\lambda_\alpha(G)$ for a graph $G$ to be $t$-disjoint path coverable without satisfying Dirac-type degree conditions.

The rest of this paper is organized as follows. In Section \ref{sec2}, we introduce some results on the clique condition $N_s(G)$ and the $\alpha$-spectral radius condition $\lambda_\alpha(G)$ for a graph $G$ to be one-to-one, one-to-many or many-to-many $t$-disjoint path coverable and provide some graphs to illustrate the sharpness. In Section \ref{sec3}, we summarize known results that are used in subsequent arguments. In Section \ref{sec4}, the proofs of our main results are given. In Section \ref{sec5}, we conclude this paper and pose a problem.

\section{Main results}\label{sec2}

In this section, we establish sufficient conditions for a graph $G$ to be one-to-one, one-to-many or many-to-many $t$-disjoint path coverable in terms of the number of $s$-cliques and the $\alpha$-spectral radius. We first define some notations as follows.
\begin{nota}
	For $i\in\{0,1,2\}$, fix $3\le \delta\le \left\lfloor \frac{n+t-i-1}{2}\right\rfloor.$ Let $F_i(n,t,\delta)=K_{\delta}\vee(K_{n+t-2\delta-i}\cup\overline{K}_{\delta-t+i})$. Denote by $f_i(n,t,\delta,s)$ the number of $s$-cliques in $F_i(n,t,\delta)$, then
	$$f_i(n,t,\delta,s)=\binom{n+t-\delta-i}{s}+(\delta-t+i)\binom{\delta}{s-1}.$$
\end{nota}

\begin{nota}
	For $i\in\{0,1,2\}$, let $h=\left\lfloor \frac{n+t-i-1}{2}\right\rfloor$, and $G_i(n,t,\delta)$ be the graph obtained from $K_h\vee(K_{n+t-2h-i}\cup\overline{K}_{h-t+i})$  by deleting $h-\delta$ edges that are incident to one common vertex in $\overline{K}_{h-t+i}$. Denote by $g_i(n,t,\delta,s)$ the number of $s$-cliques in $G_i(n,t,\delta)$, then
	$$g_i(n,t,\delta,s)=\binom{n+t-h-i}{s}+(h-t+i-1)\binom{h}{s-1}+\binom{\delta}{s-1}.$$
\end{nota}

In this paper, we first obtain the following theorem.

\begin{thm}\label{thm1.1}
Let $s$, $t$, $\delta$ be integers with $s\geq 2$ and $2\leq t\leq \delta\leq \left\lfloor \frac{n+t-3}{2} \right\rfloor $. Suppose that $G$ is a graph of order $n$ with minimum degree $\delta$. If
$$
N_s(G)> \max\{f_2(n,t,\delta,s), g_2(n,t,\delta,s)\},
$$
then $G$ is one-to-one $t$-disjoint path coverable.
\end{thm}

Let $s=2$ in Theorem \ref{thm1.1}, we have the following corollary immediately.

\begin{cor}\label{cor1.2}
	Let $t$, $\delta$ be integers with $2\leq t\leq \delta\leq \left\lfloor \frac{n+t-3}{2} \right\rfloor $. Suppose that $G$ is a graph of order $n$ with minimum degree $\delta$. If
	$$
	e(G)> \max\{f_2(n,t,\delta,2), g_2(n,t,\delta,2)\},
	$$
	then $G$ is one-to-one $t$-disjoint path coverable.
\end{cor}

Let $t=2$ in Corollary \ref{cor1.2}, we can easily obtain the following corollary.

\begin{cor}\emph{\rm{(Zhang \cite{a27})}}
	Let $\delta$ be integers with $2\leq \delta\leq \left\lfloor \frac{n-1}{2} \right\rfloor $. Suppose that $G$ is a graph of order $n$ with minimum degree $\delta$. If
	$$
	e(G)> \max\{f_2(n,2,\delta,2), g_2(n,2,\delta,2)\},
	$$
	then $G$ is Hamiltonian.
\end{cor}

Let $$f_\alpha(n,\delta,m)=\frac{\delta+\alpha n-\alpha \delta -1+\sqrt{(\delta+\alpha n-\alpha \delta -1)^2+4(1-\alpha)(2m-n\delta+\delta)}}{2}.$$

Corollary \ref{cor1.2} immediately yields the following corollary.

\begin{cor}\label{cor1.3}
	Let $n\geq 3$, $2\leq t\leq \delta\leq \left\lfloor \frac{n+t-3}{2} \right\rfloor $ and  $h_2=\max\{f_2(n,t,\delta,2), g_2(n,t,\delta,2)\}$. Suppose that $G$ is a graph of order $n$ with minimum degree $\delta$. If
	$$
	\lambda_\alpha(G)> f_\alpha(n,\delta,h_2),
	$$
	then $G$ is one-to-one $t$-disjoint path coverable.
\end{cor}

We then have the following theorem.

\begin{thm}\label{thm2.1}
	Let $s$, $t$, $\delta$ be integers with $s\geq 2$ and $3\leq t+1\leq \delta\leq \left\lfloor \frac{n+t-2}{2} \right\rfloor $. Suppose that $G$ is a graph of order $n$ with minimum degree $\delta$. If
	$$
	N_s(G)> \max\{f_1(n,t,\delta,s), g_1(n,t,\delta,s)\},
	$$
	then $G$ is one-to-many $t$-disjoint path coverable.
\end{thm}

Let $t=2$ in Theorem \ref{thm2.1}, we can easily obtain the following corollary immediately.

\begin{cor}\emph{\rm{(Zhang \cite{a28})}}
	Let $s$ and $\delta$ be integers with $s\geq 2$ and $3\leq \delta\leq \left\lfloor \frac{n}{2} \right\rfloor $. Suppose that $G$ is a graph of order $n$ with minimum degree $\delta$. If
	$$
	N_s(G)> \max\{f_1(n,2,\delta,s), g_1(n,2,\delta,s)\},
	$$
	then $G$ is Hamilton-connected.
\end{cor}

Let $s=2$ in Theorem \ref{thm2.1}. we have the following corollary immediately.

\begin{cor}\label{cor2.2}
	Let $t$, $\delta$ be integers with $3\leq t+1\leq \delta\leq \left\lfloor \frac{n+t-2}{2} \right\rfloor $. Suppose that $G$ is a graph of order $n$ with minimum degree $\delta$. If
	$$
	e(G)> \max\{f_1(n,t,\delta,2), g_1(n,t,\delta,2)\},
	$$
	then $G$ is one-to-many $t$-disjoint path coverable.
\end{cor}

Corollary \ref{cor2.2} immediately yields the following corollary.

\begin{cor}\label{cor2.3}
	Let $n\geq 3$, $3\leq t+1\leq \delta\leq \left\lfloor \frac{n+t-2}{2} \right\rfloor $ and  $h_1=\max\{f_1(n,t,\delta,2), g_1(n,t,\delta,2)\}$. Suppose that $G$ is a graph of order $n$ with minimum degree $\delta$. If
	$$
	\lambda_\alpha(G)> f_\alpha(n,\delta,h_1),
	$$
	then $G$ is one-to-many $t$-disjoint path coverable.
\end{cor}

We finally obtain the following theorem.

\begin{thm}\label{thm3.1}
	Let $s$, $t$, $\delta$ be integers with $s\geq 2$ and $4\leq 2t\leq \delta\leq \left\lfloor \frac{n+t-1}{2} \right\rfloor $. Suppose that $G$ is a graph of order $n$ with minimum degree $\delta$. If
	$$
	N_s(G)> \max\{f_0(n,t,\delta,s), g_0(n,t,\delta,s)\},
	$$
	then $G$ is many-to-many $t$-disjoint path coverable.
\end{thm}

Let $s=2$ in Theorem \ref{thm3.1}. we have the following corollary immediately.

\begin{cor}\label{cor3.2}
	Let $t$, $\delta$ be integers with $4\leq 2t\leq \delta\leq \left\lfloor \frac{n+t-1}{2} \right\rfloor $. Suppose that $G$ is a graph of order $n$ with minimum degree $\delta$. If
	$$
	e(G)> \max\{f_0(n,t,\delta,2), g_0(n,t,\delta,2)\},
	$$
	then $G$ is many-to-many $t$-disjoint path coverable.
\end{cor}

Corollary \ref{cor3.2} immediately yields the following corollary.

\begin{cor}\label{cor3.3}
	Let $n\geq 3$, $4\leq 2t\leq \delta\leq \left\lfloor \frac{n+t-1}{2} \right\rfloor $ and  $h_0=\max\{f_0(n,t,\delta,2), g_0(n,t,\delta,2)\}$. Suppose that $G$ is a graph of order $n$ with minimum degree $\delta$. If
	$$
	\lambda_\alpha(G)> f_\alpha(n,\delta,h_0),
	$$
	then $G$ is many-to-many $t$-disjoint path coverable.
\end{cor}

\noindent\textbf{Remark:} (a) For graph $F_2(n,t,\delta)$ (resp. $G_2(n,t,\delta)$), there does not exist a one-to-one $t$-disjoint path cover for any two distinct vertices in $K_{\delta-t+2}$. Thus the graph $F_2(n,t,\delta)$ (resp. $G_2(n,t,\delta)$) is not one-to-one $t$-disjoint path coverable.

(b) For graph $F_1(n,t,\delta)$ (resp. $G_1(n,t,\delta)$), there does not exist a one-to-many $t$-disjoint path cover for two disjoint subsets $U=\{u\}$ of $V(K_{\delta-t+2})$ and $V=\{v_1,v_2,\dots,v_t\}$ of $V(K_\delta )$. Thus the graph $F_1(n,t,\delta)$ (resp. $G_1(n,t,\delta)$) is not one-to-many $t$-disjoint path coverable.

(c) For graph $F_0(n,t,\delta)$ (resp. $G_0(n,t,\delta)$), there does not exist a many-to-many $t$-disjoint path cover for two disjoint subsets $U=\{u_1,u_2,\dots,u_t\}$ and $V=\{v_1,v_2,\dots,v_t\}$ of $V(K_\delta )$. Thus the graph $F_0(n,t,\delta)$ (resp. $G_0(n,t,\delta)$) is not many-to-many $t$-disjoint path coverable.

These examples show that the bounds on the number of cliques in this paper are sharp.

\section{Preliminary}\label{sec3}

The following Lemmas will be used in our later proofs.

\begin{lem}\label{lem1.1}\emph{\rm{(Lin et al. \cite{a18})}}
	Let $G$ be a graph with $n\geq t+1\geq 3$ vertices. Suppose that there exist two nonadjacent vertices $u$ and $v$ with $d_G(u)+d_G(v)\geq n+t-2$. Then $G$ is one-to-one $t$-disjoint path coverable if and only if $G+uv$ is one-to-one $t$-disjoint path coverable.
\end{lem}

\begin{lem}\label{lem2.1} \emph{\rm{(Lin et al. \cite{a19})}}
	Let $G$ be a graph with $n\geq t+1\geq 3$ vertices. Suppose that there exist two nonadjacent vertices $u$ and $v$ with $d_G(u)+d_G(v)\geq n+t-1$. Then $G$ is one-to-many $t$-disjoint path coverable if and only if $G+uv$ is one-to-many $t$-disjoint path coverable.
\end{lem}

\begin{lem}\label{lem3.1} \emph{\rm{(Lin et al. \cite{a19})}}
	Let $G$ be a graph with $n\geq 2t\geq 2$ vertices. Suppose that there exist two nonadjacent vertices $u$ and $v$ with $d_G(u)+d_G(v)\geq n+t$. Then $G$ is many-to-many $t$-disjoint path coverable if and only if $G+uv$ is many-to-many $t$-disjoint path coverable.
\end{lem}

\begin{lem}\label{lem1}
	Let $n,t,x,s$ be non-negative integers with $x\geq s-1\geq 1$. For $i\in \{0,1,2\}$ and $t\geq 2$ with $t-i+1\leq x\leq \left\lfloor \frac{n+t-i-1}{2}\right\rfloor$, the function
	$$\lambda_i(n,t,s,x)=(x-t+i-1)\binom{x}{s-1}+\binom{n+t-x-i}{s}$$
	is convex.
\end{lem}

\noindent\textbf{Proof:}
For $i\in \{0,1,2\}$ and $t\geq 2$ with $t-i+1\leq x\leq \left\lfloor \frac{n+t-i-1}{2}\right\rfloor$, recall that for the combinatorial function $\binom{z}{k}$ (extended to real $z$), its first and second derivatives are
$$
\frac{\mathrm{d}}{\mathrm{d}z}\binom{z}{k} = \binom{z}{k} \sum_{j=0}^{k-1} \frac{1}{z-j}
$$
and
$$
\frac{\mathrm{d}^2}{\mathrm{d}z^2}\binom{z}{k} = \binom{z}{k} \left[ \left( \sum_{j=0}^{k-1} \frac{1}{z-j} \right)^2 - \sum_{j=0}^{k-1} \frac{1}{(z-j)^2} \right].
$$

First, consider $f(x)=(x - t + i - 1)\binom{x}{s - 1}$. Then
\begin{equation*}
		\frac{\mathrm{d}^2 f(x)}{\mathrm{d}x^2} = \binom{x}{s-1} \left[ (x-t+i-1) \left( \left( \sum_{j=0}^{s-2} \frac{1}{x-j} \right)^2 - \sum_{j=0}^{s-2} \frac{1}{(x-j)^2} \right) + 2\sum_{j=0}^{s-2} \frac{1}{x-j} \right] > 0.
\end{equation*}

Next, let $g(x)=\binom{n + t - x - i}{s}$ and set $y = n + t - i - x$. Then
$$
\frac{\mathrm{d}^2 g(x)}{\mathrm{d}x^2} = \binom{y}{s} \left[ \left( \sum_{j=0}^{s-1} \frac{1}{y-j} \right)^2 - \sum_{j=0}^{s-1} \frac{1}{(y-j)^2} \right] > 0.
$$

Since the second derivative of $\lambda_i$ is
$$
\frac{\mathrm{d}^2\lambda_i}{\mathrm{d}x^2} = \frac{\mathrm{d}^2 f}{\mathrm{d}x^2} + \frac{\mathrm{d}^2 g}{\mathrm{d}x^2} > 0,
$$
the function $\lambda_i(n,t,s,x)$ is convex. \qed

\begin{lem}\label{lem0} \emph{\rm{(Wu, Xue and Kang \cite{a29})}}
	Let $G$ be a simple connected graph of order $n$ and minimum degree $\delta$ with $m$ edges and $0\leq \alpha< 1$. Then
	$$
	\lambda_\alpha(G)\leq f_\alpha(n,\delta,m).
	$$
\end{lem}

To prove our results, we also need a definition from Kopylov \cite{a37}.

\begin{df}\emph{\rm{($h$-disintegration of a graph, Kopylov \cite{a37})}}
	Let $G$ be a graph and $h$ be a positive integer. Delete all vertices of degree at most $h$ from $G$; for the resulting graph $G'$, we again delete all vertices of degree at most $h$ from $G'$. Iterating this process until we finally obtain a graph, denoted by $D(G;h)$, such that either $D(G;h)$ is a null graph or $\delta(D(G;h))\geq h+1$. The graph $D(G;h)$ is called the $(h+1)$-core of $G$.
\end{df}

\section{ Proofs of main results }\label{sec4}

\subsection{Proof of Theorem \ref{thm1.1}}
\noindent\textbf{Proof:}
By contradiction, suppose that the graph $G$ of order $n$ with minimum degree $\delta$ is not one-to-one $t$-disjoint path coverable. Then $G$ is not a complete graph. Otherwise, $\delta(G)=n-1\geq \frac{n+t-2}{2}$. By Theorem \ref{thm0}(i), the graph $G$ is one-to-one $t$-disjoint path coverable, a contradiction.

Let $w$ be a vertex of $G$ with $d_G(w)=\delta$. Suppose that the induced subgraph $G[V(G)\backslash \{w\}]=K_{n-1}$. Then, $\sigma_2(G)=\delta+n-2\geq n+t-2$. By Lemma \ref{lem1.1}, the graph $G$ is one-to-one $t$-disjoint path coverable, a contradiction. Thus $G[V(G)\backslash \{w\}]\neq K_{n-1}$.

If there exist two vertices $u,v\in V(G)\backslash \{w\}$ such that $uv\notin E(G)$ and $d_G(u)+d_G(v)\geq n+t-2$, we denote by $G_1=G+uv$. Then $G_1[V(G_1)\backslash\{w\}]\neq K_{n-1}$. Otherwise, $\sigma_2(G_1)=\delta+n-2\geq n+t-2$, then by Lemma \ref{lem1.1} the graph $G$ is one-to-one $t$-disjoint path coverable, a contradiction. Iterating this process until no such pair of vertices remains. We finally obtain a graph, denoted by $Q$.

If the graph $Q$ is a complete graph, then by Lemma \ref{lem1.1}, the graph $G$ is one-to-one $t$-disjoint path coverable, a contradiction. Thus, we assume that the graph $Q$ is not a complete graph. For any $x,y\in V(Q)\backslash \{w\}$ with $xy\notin E(G)$, $d_Q(x)+d_Q(y)\leq n+t-3$. Since edges are only added in $G[V(G)\backslash\{w\}]$, $\delta(Q)=d_G(w)=\delta$. By Lemma \ref{lem1.1}, the graph $Q$ is not one-to-one $t$-disjoint path coverable. Let $h=\left\lfloor \frac{n+t-3}{2} \right\rfloor$, $D=D(Q;h)$ and $d=|D|$. We distinguish the following two cases.

\par {\bfseries Case 1:}\ $d=0$.

In the $h$-disintegration process, let $Q_0=Q$ and $Q_{i+1}=Q_i-x_i$, $0\leq i\leq n-1$, where $x_i$ is a vertex with $d_{Q_{i}}(x_i)\leq h$. Then $\delta(Q)\leq h$; otherwise no vertex could be deleted. Without loss of generality, we choose $w$ to be $x_0$, the vertex with minimum degree in $Q$.

By the definition of $h$-disintegration, $d_{Q_i}(x_i)\leq h$, $0\leq i\leq n-h-1$. Then,
\begin{equation*}
	\begin{split}
		N_s(Q)\leq \binom{\delta}{s-1}+(n-h-1)\binom{h}{s-1}+\binom{h}{s}.
	\end{split}
\end{equation*}

We distinguish the following two subcases.

\par {\bfseries Subcase 1.1:}\ $n+t$ is odd.

Since $n+t\geq 5$, we define $n+t=2k+1$ for some integer $k\geq 2$. Since $h=\left\lfloor \frac{n+t-3}{2} \right\rfloor=k-1$,
\begin{equation*}
\begin{split}
    N_s(Q)-g_2(n,t,\delta,s)\leq&(n-k)\binom{k-1}{s-1}+\binom{k-1}{s}-(k-t)\binom{k-1}{s-1}-\binom{k}{s}\\
		=&\binom{k-1}{s-1}+\binom{k-1}{s}-\binom{k}{s}=0.\\
\end{split}	
\end{equation*}

\par {\bfseries Subcase 1.2:}\ $n+t$ is even.

Since $n+t\geq 5$, we define $n+t=2k$ for some integer $k\geq 3$. Since $h=\left\lfloor \frac{n+t-3}{2} \right\rfloor=k-2$,
\begin{equation*}
	\begin{split}
		N_s(Q)-g_2(n,t,\delta,s)\leq&(n-k+1)\binom{k-2}{s-1}+\binom{k-2}{s}-(k-t-1)\binom{k-2}{s-1}-\binom{k}{s}\\
			=&2\binom{k-2}{s-1}+\binom{k-2}{s}-\binom{k}{s}=\binom{k-2}{s-1}-\binom{k-1}{s-1}\leq 0.\\			
	\end{split}	
\end{equation*}

Then, $\max\{f_2(n,t,\delta,s), g_2(n,t,\delta,s)\}<N_s(G)\leq N_s(Q)\leq g_2(n,t,\delta,s)$, a contradiction.
			
\par {\bfseries Case 2:}\ $d\neq 0$.
			
\begin{cla}\label{cla1.1}
	$D$ is a complete graph.
\end{cla}
\noindent\textbf{Proof:} We establish the claim by contradiction. Suppose that there exist two vertices $u,v\in V(D)$ such that $uv\notin E(D)$. By the definition of $h$-disintegration, $d_D(u),d_D(v)\geq h+1$. Since $uv\notin E(D)$, $n+t-2\leq d_D(u)+d_D(v)\leq d_Q(u)+d_Q(v)\leq n+t-3$, a contradiction. Thus, the graph $D$ is a complete graph. $\hfill\Box$

\begin{cla}\label{cla1.2}
	$\delta\leq n+t-d-2$.
\end{cla}			
\noindent\textbf{Proof:} By contradiction, suppose that $\delta \geq n+t-d-1$. Then $d_D(u)=d-1\geq n+t-\delta-2$ for all $u\in V(D)$. Since the graph $D$ is a complete graph and $d_D(u)\geq h+1$ for all $u\in V(D)$, $d\geq h+2$. Then, for any vertex $v\in V(Q)\backslash V(D)$, there exist at least two vertices in $D$, which are not adjacent to $v$ in $Q$. Let $x\in V(Q)\backslash V(D)$, we assume that $y\in V(D)$ and $xy\notin E(Q)$. Note that $w\in V(Q)\backslash V(D)$. We distinguish two cases.
			
Suppose first that $V(Q)\backslash V(D)=\{w\}$. Then $x=w$ and $|D|=n-1$. If there exists a vertex $v\in V(D)$ with $vw\notin E(Q)$, then $n+t-2\leq n-2+\delta \leq d_Q(v)+d_Q(w)\leq n+t-3$, a contradiction. Then, the graph $Q$ is a complete graph, a contradiction.
			
Suppose next that $V(Q)\backslash V(D)\neq \{w\}$. Then there exists a vertex $x\in V(Q)\backslash (V(D)\cup \{w\})$. Since $d_Q(x)\geq \delta$ and $y\in V(D)$, $d_Q(x)+d_Q(y)\geq \delta+n+t-\delta-2=n+t-2$. Since $xy\notin E(Q)$, it follows that $n+t-2\leq d_Q(x)+d_Q(y)\leq n+t-3$, a contradiction. Then, $\delta\leq n+t-d-2$.     $\hfill\Box$
			
Let $D'$ be the $(n+t-d-1)$-core of $Q$. Since $d\geq h+2$, $n+t-d-2\leq n+t-h-4\leq h$. Therefore, $D\subseteq D'$. We distinguish the following two subcases.
			
\par {\bfseries Subcase 2.1:}\ $D'=D$.

Since $D'=D$, $|D'|=|D|=d$. Then, by the definition of $(n+t-d-2)$-disintegration,
\begin{equation*}
	\begin{split}
		N_s(Q)\leq&\binom{\delta}{s-1}+(n-d-1)\binom{n+t-d-2}{s-1}+\binom{d}{s}\\
				=&\binom{\delta}{s-1}+\lambda_2(n,t,s,n+t-d-2).\\
	\end{split}	
\end{equation*}

Suppose first that $n+t-d-2\leq s-2$. By Claim \ref{cla1.2}, $\delta\leq n+t-d-2\leq s-2$. Then
$$
\max\{f_2(n,t,\delta,s), g_2(n,t,\delta,s)\}<N_s(G)\leq \binom{d}{s}\leq \binom{n+t-\delta-2}{s}\leq f_2(n,t,\delta,s),
$$ a contradiction.

Suppose next that $n+t-d-2\geq s-1$. By Lemma \ref{lem1}, the function $\lambda_2(n,t,s,x)$ is convex for $t\leq x\leq h$. Since $t-1\leq \delta\leq n+t-d-2\leq h$, $\max\{f_2(n,t,\delta,s), g_2(n,t,\delta,s)\}<N_s(G)\leq N_s(Q)\leq \max\{f_2(n,t,\delta,s), g_2(n,t,\delta,s)\}$, a contradiction.
									
\par {\bfseries Subcase 2.2:}\ $D'\neq D$.
									
Let $u\in V(D')\backslash V(D)$. Since $d\geq h+2$, there exist at least two vertices in $D$, which are not adjacent to $u$. Let $v\in V(D)$ be a vertex such that $uv\notin E(Q)$. By the definition of $(n+t-d-2)$-disintegration and the complete graph $D$, $n+t-2=n+t-d-1+d-1\leq d_Q(u)+d_Q(v)\leq n+t-3$, a contradiction.
									
This completes the proof of Theorem \ref{thm1.1}. $\hfill\Box$

\subsection{Proof of Corollary \ref{cor1.3}}
\noindent\textbf{Proof:}
We first claim that $e(G)>\max\{f_2(n,t,\delta,2), g_2(n,t,\delta,2)\}$. Otherwise, $e(G)\leq \max\{f_2(n,t,\delta,2), g_2(n,t,\delta,2)\}$. By Lemma \ref{lem0},
\begin{equation*}
	\begin{split}
		\lambda_\alpha(G)\leq& \frac{\delta+\alpha n-\alpha \delta -1+\sqrt{(\delta+\alpha n-\alpha \delta -1)^2+4(1-\alpha)(2h_2-n\delta+\delta)}}{2}\\
		\leq& f_\alpha(n,\delta,h_2),
	\end{split}
\end{equation*}
a contradiction. Then $e(G)>\max\{f_2(n,t,\delta,2), g_2(n,t,\delta,2)\}$. By Corollary \ref{cor1.2}, the graph $G$ is one-to-one $t$-disjoint path coverable.  $\hfill\Box$

\subsection{Proof of Theorem \ref{thm2.1}}
\noindent\textbf{Proof:}
By contradiction, suppsoe that the graph $G$ of order $n$ with minimum degree $\delta$ is not one-to-many $t$-disjoint path coverable. Then $G$ is not a complete graph. Otherwise, $\delta(G)=n-1\geq \frac{n+t-1}{2}$. By Theorem \ref{thm0}(ii), the graph $G$ is one-to-many $t$-disjoint path coverable, a contradiction.

Let $w$ be a vertex of $G$ with $d_G(w)=\delta$. Suppose that the induced subgraph $G[V(G)\backslash \{w\}]=K_{n-1}$. Then, $\sigma_2(G)=\delta+n-2\geq n+t-1$. By Lemma \ref{lem2.1}, the graph $G$ is one-to-many $t$-disjoint path coverable, a contradiction. Thus $G[V(G)\backslash \{w\}]\neq K_{n-1}$.

If there exist two vertices $u,v\in V(G)\backslash \{w\}$ such that $uv\notin E(G)$ and $d_G(u)+d_G(v)\geq n+t-1$, we denote by $G_1=G+uv$. Then $G[V(G_1)\backslash \{w\}]\neq K_{n-1}$. Otherwise, $\sigma_2(G_1)=\delta+n-2\geq n+t-1$, then by Lemma \ref{lem2.1}, the graph $G$ is one-to-many $t$-disjoint path coverable, a contradiction. Iterating this process until no such pair of vertices remains. Then we obtain a graph, denoted by $Q$.

If the graph $Q$ is a complete graph, then by Lemma \ref{lem2.1}, the graph $G$ is one-to-many $t$-disjoint path coverable, a contradiction. Thus, we assume that the graph $Q$ is not a complete graph. For any $x,y\in V(Q)\backslash \{w\}$ with $xy\notin E(G)$, $d_Q(x)+d_Q(y)\leq n+t-2$. Since edges are only added in $G[V(G)\backslash\{w\}]$, $\delta(Q)=d_G(w)=\delta$. By Lemma \ref{lem2.1}, the graph $Q$ is not one-to-many $t$-disjoint path coverable. Let $h=\left\lfloor \frac{n+t-2}{2} \right\rfloor$, $D=D(Q;h)$ and $d=|D|$. We distinguish the following two cases.

\par {\bfseries Case 1:}\ $d=0$.

In the $h$-disintegration process, let $Q_0=Q$ and $Q_{i+1}=Q_i-x_i$, $0\leq i\leq n-1$, where $x_i$ is a vertex with $d_{Q_{i}}(x_i)\leq t$. Then $\delta(Q)\leq h$; otherwise no vertex could be deleted. Without loss of generality, we choose $w$ to be $x_0$, the vertex with minimum degree in $Q$.

By the definition of $h$-disintegration, $d_{Q_i}(x_i)\leq h$, $0\leq i\leq n-h-1$. Then,
\begin{equation*}
	\begin{split}
		N_s(Q)\leq \binom{\delta}{s-1}+(n-h-1)\binom{h}{s-1}+\binom{h}{s}.
	\end{split}
\end{equation*}

We distinguish the following two subcases.

\par {\bfseries Subcase 1.1:}\ $n+t$ is odd.

Since $n+t\geq 5$, we define $n+t=2k+1$ for some integer $k\geq 2$. Since $h=\left\lfloor \frac{n+t-2}{2} \right\rfloor=k-1$,
\begin{equation*}
	\begin{split}
		N_s(Q)-g_1(n,t,\delta,s)\leq&(n-k)\binom{k-1}{s-1}+\binom{k-1}{s}-(k-t-1)\binom{k-1}{s-1}-\binom{k+1}{s}\\
		=&2\binom{k-1}{s-1}+\binom{k-1}{s}-\binom{k+1}{s}=\binom{k-1}{s-1}-\binom{k}{s-1}\leq 0.
	\end{split}	
\end{equation*}

\par {\bfseries Subcase 1.2:}\ $n+t$ is even.

Since $n+t\geq 5$, we define $n+t=2k$ for some integer $k\geq 3$. Since $h=\left\lfloor \frac{n+t-2}{2} \right\rfloor=k-1$,
\begin{equation*}
	\begin{split}
		N_s(Q)-g_1(n,t,\delta,s)\leq&(n-k)\binom{k-1}{s-1}+\binom{k-1}{s}-(k-t-1)\binom{k-1}{s-1}-\binom{k}{s}\\
		=&2\binom{k-1}{s-1}+\binom{k-1}{s}-\binom{k}{s}=0.			
	\end{split}	
\end{equation*}

Then, $\max\{f_1(n,t,\delta,s), g_1(n,t,\delta,s)\}<N_s(G)\leq N_s(Q)\leq g_1(n,t,\delta,s)$, a contradiction.

\par {\bfseries Case 2:}\ $d\neq 0$.
			
\begin{cla}\label{cla2.1}
    $D$ is a complete graph.
\end{cla}
			
\noindent\textbf{Proof:} We establish the claim by contradiction. Suppose that there exist two vertices $u,v\in V(D)$ such that $uv\notin E(D)$. By the definition of $h$-disintegration, $d_D(u),d_D(v)\geq h+1$. Since $uv\notin E(D)$, $n+t-1\leq d_D(u)+d_D(v)\leq d_Q(u)+d_Q(v)\leq n+t-2$, a contradiction. Thus, the graph $D$ is a complete graph. $\hfill\Box$
			
\begin{cla}\label{cla2.2}
	$\delta\leq n+t-d-1$.
\end{cla}
\noindent\textbf{Proof:} By contradiction, suppose that $\delta \geq n+t-d$. Then $d_D(u)=d-1\geq n+t-\delta-1$ for all $u\in V(D)$. Since the graph $D$ is a complete graph and $d_D(u)\geq h+1$ for all $u\in V(D)$, $d\geq h+2$. Then, for any vertex $v\in V(Q)\backslash V(D)$, there exist at least two vertices in $D$, which are not adjacent to $v$ in $Q$. Let $x\in V(Q)\backslash V(D)$, we assume that $y\in V(D)$ and $xy\notin E(Q)$. Note that $w\in V(Q)\backslash V(D)$. We distinguish two cases.

Suppose first that $V(Q)\backslash V(D)=\{w\}$. Then $x=w$ and $|D|=n-1$. If there exists a vertex $v\in V(D)$ with $vw\notin E(Q)$, then $n+t-1\leq n-2+\delta \leq d_Q(v)+d_Q(w)\leq n+t-2$, a contradiction. Then, the graph $Q$ is a complete graph, a contradiction.

Suppose next that $V(Q)\backslash V(D)\neq \{w\}$. Then there exists a vertex $x\in V(Q)\backslash (V(D)\cup \{w\})$. Since $d_Q(x)\geq \delta$ and $y\in V(D)$, $d_Q(x)+d_Q(y)\geq \delta+n+t-\delta-1=n+t-1$. Since $xy\notin E(G)$, it follows that $n+t-1\leq d_Q(x)+d_Q(y)\leq n+t-2$, a contradiction. Then, $\delta\leq n+t-d-1$.  $\hfill\Box$
			
Let $D'$ be the $(n+t-d)$-core of $Q$. Since $d\geq h+2$, $n+t-d-1\leq n+t-h-3\leq h$. Therefore, $D\subseteq D'$. We distinguish the following two subcases.
			
\par {\bfseries Subcase 2.1:}\ $D'=D$.
			
Since $D'=D$, $|D'|=|D|=d$. Then, by the definition of $(n+t-d-1)$-disintegration,
\begin{equation*}
	\begin{split}
		N_s(Q)\leq&\binom{\delta}{s-1}+(n-d-1)\binom{n+t-d-1}{s-1}+\binom{d}{s}\\
			  	 =&\binom{\delta}{s-1}+\lambda_1(n,t,s,n+t-d-1).
	\end{split}
\end{equation*}

Suppose first that $n+t-d-1\leq s-2$. By Claim \ref{cla2.2}, $\delta\leq n+t-d-1\leq s-2$. Then
$$
\max\{f_1(n,t,\delta,s), g_1(n,t,\delta,s)\}<N_s(G)\leq \binom{d}{s}\leq \binom{n+t-\delta-1}{s}\leq f_1(n,t,\delta,s),
$$ a contradiction.

Suppose next that $n+t-d-1\geq s-1$. By Lemma \ref{lem1}, the function $\lambda_1(n,t,s,x)$ is convex for $t\leq x\leq h$. Since $t+1\leq \delta\leq n+t-d-1\leq h$, $\max\{f_1(n,t,\delta,s), g_1(n,t,\delta,s)\}<N_s(G)\leq N_s(Q)\leq \max\{f_1(n,t,\delta,s), g_1(n,t,\delta,s)\}$, a contradiction.

\par {\bfseries Subcase 2.2:}\ $D'\neq D$.
									
Let $u\in V(D')\backslash V(D)$. Since $d\geq h+2$, there exist at least two vertices in $D$, which are not adjacent to $u$. Let $v\in V(D)$ be a vertex such that $uv\notin E(Q)$. By the definition of $(n+t-d-1)$-disintegration and the complete graph $D$, $d_Q(u)+d_Q(v)\geq n+t-d+d-1=n+t-1$. Since $uv\notin E(Q)$, $n+t-1\leq d_Q(u)+d_Q(v)\leq n+t-2$, a contradiction.
									
This completes the proof of Theorem \ref{thm2.1}. $\hfill\Box$

\subsection{Proof of Corollary \ref{cor2.3}}
\noindent\textbf{Proof:}
We first claim that $e(G)>\max\{f_1(n,t,\delta,2), g_1(n,t,\delta,2)\}$. Otherwise, $e(G)\leq \max\{f_1(n,t,\delta,2), g_1(n,t,\delta,2)\}$. By Lemma \ref{lem0},
\begin{equation*}
	\begin{split}
		\lambda_\alpha(G)\leq& \frac{\delta+\alpha n-\alpha \delta -1+\sqrt{(\delta+\alpha n-\alpha \delta -1)^2+4(1-\alpha)(2h_1-n\delta+\delta)}}{2}\\
		\leq& f_\alpha(n,\delta,h_1),
	\end{split}
\end{equation*}
a contradiction. Then $e(G)>\max\{f_1(n,t,\delta,2), g_1(n,t,\delta,2)\}$. By Corollary \ref{cor2.2}, the graph $G$ is one-to-many $t$-disjoint path coverable.  $\hfill\Box$

\subsection{Proof of Theorem \ref{thm3.1}}
\noindent\textbf{Proof:}
By contradiction, suppose that the graph $G$ of order $n$ and minimum degree $\delta$ is not many-to-many $t$-disjoint path coverable. Then $G$ is not a complete graph. Otherwise, $\delta(G)=n-1\geq \frac{n+t}{2}$. By Theorem \ref{thm0}(iii), the graph $G$ is many-to-many $t$-disjoint path coverable, a contradiction.

Let $w$ be a vertex of $G$ with $d_G(w)=\delta$. Suppose that the induced subgraph $G[V(G)\backslash \{w\}]=K_{n-1}$. Then, $\sigma_2(G)=\delta+n-2\geq n+2t-2\geq n+t$. By Lemma \ref{lem3.1}, the graph $G$ is many-to-many $t$-disjoint path coverable, a contradiction. Thus $G[V(G)\backslash \{w\}]\neq K_{n-1}$.

If there exist two vertices $u,v\in V(G)\backslash \{w\}$ such that $uv\notin E(G)$ and $d_G(u)+d_G(v)\geq n+t$, we denote by $G_1=G+uv$. Then $G[V(G_1)\backslash \{w\}]\neq K_{n-1}$. Otherwise, $\sigma_2(G_1)\geq \delta+n-2\geq n+2t-2\geq n+t$, then by Lemma \ref{lem3.1}, the graph $G$ is many-to-many $t$-disjoint path coverable, a contradiction. Iterating this process until no such pair of vertices remains. Then we obtain a graph, denoted by $Q$.

If the graph $Q$ is a complete graph, then by Lemma \ref{lem3.1}, the graph $G$ is many-to-many $t$-disjoint path coverable, a contradiction. Thus, we assume that the graph $Q$ is not a complete graph. For any $x,y\in V(Q)\backslash \{w\}$ with $xy\notin E(G)$, $d_Q(x)+d_Q(y)\leq n+t-1$. Since edges are only added in $G[V(G)\backslash\{w\}]$, $\delta(Q)=d_G(w)=\delta$. By Lemma \ref{lem3.1}, the graph $Q$ is not many-to-many $t$-disjoint path coverable. Let $h=\left\lfloor \frac{n+t-1}{2} \right\rfloor$, $D=D(Q;h)$ and $d=|D|$. We distinguish the following two cases.

\par {\bfseries Case 1:}\ $d=0$.

In the $h$-disintegration process, let $Q_0=Q$ and $Q_{i+1}=Q_i-x_i$, $0\leq i\leq n-1$, where $x_i$ is a vertex such that $d_{Q_{i}}(x_i)\leq t$. Then $\delta(Q)\leq h$; otherwise no vertex could be deleted. Without loss of generality, we choose $w$ to be $x_0$, the vertex with minimum degree in $Q$.

By the definition of $h$-disintegration, $d_{Q_i}(x_i)\leq h$, $0\leq i\leq n-h-1$. Then,
\begin{equation*}
	\begin{split}
		N_s(Q)\leq \binom{\delta}{s-1}+(n-h-1)\binom{h}{s-1}+\binom{h}{s}.
	\end{split}
\end{equation*}

We distinguish the following two subcases.

\par {\bfseries Subcase 1.1:}\ $n+t$ is odd.

Since $n+t\geq 5$, we define $n+t=2k+1$ for some integer $k\geq 2$. Since $h=\left\lfloor \frac{n+t-1}{2} \right\rfloor=k$,
\begin{equation*}
	\begin{split}
		N_s(Q)-g_0(n,t,\delta,s)\leq&(n-k-1)\binom{k}{s-1}+\binom{k}{s}-(k-t-1)\binom{k}{s-1}-\binom{k+1}{s}\\
		=&\binom{k}{s-1}+\binom{k}{s}-\binom{k+1}{s}=0.
	\end{split}	
\end{equation*}

\par {\bfseries Subcase 1.2:}\ $n+t$ is even.

Since $n+t\geq 5$, we define $n+t=2k$ for some integer $k\geq 3$. Since $h=\left\lfloor \frac{n+t-1}{2} \right\rfloor=k-1$,
\begin{equation*}
	\begin{split}
		N_s(Q)-g_0(n,t,\delta,s)\leq&(n-k)\binom{k-1}{s-1}+\binom{k-1}{s}-(k-t-2)\binom{k-1}{s-1}-\binom{k+1}{s}\\
		=&2\binom{k-1}{s-1}+\binom{k-1}{s}-\binom{k+1}{s}=\binom{k-1}{s-1}-\binom{k}{s-1}\leq 0.			
	\end{split}	
\end{equation*}

Then, $\max\{f_0(n,t,\delta,s), g_0(n,t,\delta,s)\}<N_s(G)\leq N_s(Q)\leq g_0(n,t,\delta,s)$, a contradiction.

\par {\bfseries Case 2:}\ $d\neq 0$.
			
\begin{cla}\label{cla3.1}
$D$ is a complete graph.
\end{cla}
\noindent\textbf{Proof:} We establish the claim by contradiction. Suppose that there exist two vertices $u,v\in V(D)$ such that $uv\notin E(D)$. By the definition of $h$-disintegration, $d_D(u),d_D(v)\geq h+1$. Since $uv\notin E(D)$, $n+t\leq d_D(u)+d_D(v)\leq d_Q(u)+d_Q(v)\leq n+t-1$, a contradiction. Thus, the graph $D$ is a complete graph. $\hfill\Box$
			
\begin{cla}\label{cla3.2}
	$\delta\leq n+t-d$.
\end{cla}
\noindent\textbf{Proof:} By contradiction, suppose that $\delta \geq n+t-d+1$. Then $d_D(u)=d-1\geq n+t-\delta$ for all $u\in V(D)$. Since the graph $D$ is a complete graph and $d_D(u)\geq h+1$ for all $u\in V(D)$, $d\geq h+2$. Then, for any vertex $v\in V(Q)\backslash V(D)$, there exist at least two vertices in $D$, which are not adjacent to $v$ in $Q$. Let $x\in V(Q)\backslash V(D)$, we assume that $y\in V(D)$ and $xy\notin E(Q)$. Note that $w\in V(Q)\backslash V(D)$. We distinguish two cases.

Suppose first that $V(Q)\backslash V(D)=\{w\}$. Then $x=w$ and $|D|=n-1$. If there exists a vertex $v\in V(D)$ with $vw\notin E(Q)$, then $n+t\leq n+2t-2\leq n-2+\delta \leq d_Q(v)+d_Q(w)\leq n+t-1$, a contradiction. Then, the graph $Q$ is a complete graph, a contradiction.

Suppose next that $V(Q)\backslash V(D)\neq \{w\}$. Then there exists a vertex $x\in V(Q)\backslash (V(D)\cup \{w\})$. Since $d_Q(x)\geq \delta$ and $y\in V(D)$, $d_Q(x)+d_Q(y)\geq \delta+n+t-\delta=n+t$. Since $xy\notin E(G)$, it follows that $n+t\leq d_Q(x)+d_Q(y)\leq n+t-1$, a contradiction. Then, $\delta\leq n+t-d$.     $\hfill\Box$
			
Let $D'$ be the $(n+t-d+1)$-core of $Q$. Since $d\geq h+2$, $n+t-d\leq n+t-h-2\leq h$. Therefore, $D\subseteq D'$. We distinguish the following two subcases.
			
\par {\bfseries Subcase 2.1:}\ $D'=D$.
			
Since $D'=D$, $|D'|=|D|=d$. Then, by the definition of $(n+t-d)$-disintegration,
\begin{equation*}
	\begin{split}
		N_s(Q)\leq&\binom{\delta}{s-1}+(n-d-1)\binom{n+t-d}{s-1}+\binom{d}{s}\\
				=&\binom{x}{s-1}+\lambda_0(n,t,s,n+t-d).
	\end{split}	
\end{equation*}

Suppose first that $n+t-d\leq s-2$. By Claim \ref{cla3.2}, $\delta\leq n+t-d\leq s-2$. Then
$$
\max\{f_0(n,t,\delta,s), g_0(n,t,\delta,s)\}<N_s(G)\leq \binom{d}{s}\leq \binom{n+t-\delta}{s}\leq f_0(n,t,\delta,s),
$$ a contradiction.

Suppose next that $n+t-d\geq s-1$. By Lemma \ref{lem1}, the function $\lambda_0(n,t,s,x)$ is convex for $t+1\leq x\leq h$. Since $t+1\leq 2t\leq \delta\leq n+t-d\leq h$, $\max\{f_0(n,t,\delta,s), g_0(n,t,\delta,s)\}<N_s(G)\leq N_s(Q)\leq \max\{f_0(n,t,\delta,s), g_0(n,t,\delta,s)\}$, a contradiction.

\par {\bfseries Subcase 2.2:}\ $D'\neq D$.
									
Let $u\in V(D')\backslash V(D)$. Since $d\geq h+2$, there exist at least two vertices in $D$, which are not adjacent to $u$. Let $v\in V(D)$ be vertex such that $uv\notin E(Q)$. By the definition of $(n+t-d)$-disintegration and the complete graph $D$, $n+t=n+t-d+1+d-1\leq d_Q(u)+d_Q(v)\leq n+t-1$, a contradiction.
									
This completes the proof of Theorem \ref{thm3.1}. $\hfill\Box$

\subsection{Proof of Corollary \ref{cor3.3}}
\noindent\textbf{Proof:}
We first claim that $e(G)>\max\{f_0(n,t,\delta,2), g_0(n,t,\delta,2)\}$. Otherwise, $e(G)\leq \max\{f_0(n,t,\delta,2), g_0(n,t,\delta,2)\}$. By Lemma \ref{lem0},
\begin{equation*}
	\begin{split}
		\lambda_\alpha(G)\leq& \frac{\delta+\alpha n-\alpha \delta -1+\sqrt{(\delta+\alpha n-\alpha \delta -1)^2+4(1-\alpha)(2h_0-n\delta+\delta)}}{2}\\
		\leq& f_\alpha(n,\delta,h_0),
	\end{split}
\end{equation*}
a contradiction. Then $e(G)>\max\{f_0(n,t,\delta,2), g_0(n,t,\delta,2)\}$. By Corollary \ref{cor3.2}, the graph $G$ is many-to-many $t$-disjoint path coverable.  $\hfill\Box$

\section{ Conclusion }\label{sec5}

In this paper, using the $h$-disintegration \cite{a37} and the $\alpha$-spectral radius formula \cite{a29}, we characterized the minimum number of $s$-cliques and the $\alpha$-spectral radius that guarantee  a graph is one-to-one, one-to-many, or many-to-many $t$-disjoint path coverable, respectively. Specifically, Park, Kim and Lim \cite{a9} showed that if a graph $G$ of order $n\geq 2t+1$ is many-to-many $t$-disjoint path coverable, then $\delta(G)\geq t+1$. Thus, we pose the following problem.

\begin{prob}
	Let $s$, $t$, $\delta$ be integers with $s\geq 2$ and $3\leq t+1\leq \delta\leq 2t-1$. Characterize the minimum number of $s$-cliques that guarantees a graph of order $n$ with minimum degree $\delta$ is many-to-many $t$-disjoint path coverable.
\end{prob}

\section*{ Declaration of competing interest }
There is no conflict of interest.

\section*{ Data availability }
No data was used for the research described in the paper.

\section*{ Acknowledgements }
This work was partly supported by the National Natural Science Foundation of China (No. 12101126), Natural Science Foundation of Fujian Province (No. 2023J01539). This work was also partly supported  China Scholarship Council (No. 202409100010).

\end{CJK}
\end{document}